%% file: main.tex
\documentclass[11pt,a4paper]{article}
\usepackage[lmargin=1.0in,rmargin=1.0in,bottom=1.0in,top=1.0in,twoside=False]{geometry}
\usepackage{mathtools}
\usepackage{hchang}

\let\qedhere\relax

\usepackage{amsthm}
\usepackage{graphicx}
\usepackage{enumerate}
\usepackage{tikz}
\usepackage{stackengine}
\usepackage{accents}
\usepackage{textpos}
\usetikzlibrary{shapes}

\usepackage[T1]{fontenc}

\usepackage{enumitem}

\usepackage{microtype}
\usepackage{comment}
\usepackage[capitalize,nameinlink]{cleveref}

\usepackage{stmaryrd}

\renewcommand{\geq}{\geqslant}

\renewcommand{\leq}{\leqslant}

\newtheorem{lemma}{Lemma}
\newtheorem{theorem}{Theorem}

\newtheorem{conjecture}{Conjecture}

\newtheorem{claim}{Claim}
\newcommand{\Oh}{\mathcal{O}}
\newcommand{\dist}{\mathsf{dist}}

\newcommand{\real}{\mathbb{R}}

\newcommand{\wei}{\mathbf{w}}
\newcommand{\Rp}{\real_{\geq 0}}
\newcommand{\Xx}{\mathcal{X}}
\newcommand{\Uu}{\mathcal{U}}
\newcommand{\Zz}{\mathcal{Z}}
\newcommand{\Ss}{\mathcal{S}}
\newcommand{\Pp}{\mathcal{P}}
\newcommand{\Dd}{\mathfrak{D}}
\newcommand{\af}[1]{\textcolor{black!50}{#1}}

\newcommand{\wh}[1]{\widehat{#1}}

\author{
	    \'Edouard Bonnet\thanks{Édouard Bonnet was supported by the French National Research Agency through the project TWIN-WIDTH with reference number ANR-21-CE48-0014.}\\[0.05cm] 
        \normalsize\af{CNRS, ENS de Lyon,}\\
        \normalsize\af{Universit\'e Claude Bernard Lyon 1,}\\
        \normalsize\af{LIP, UMR 5668}\\
        \normalsize\href{mailto:edouard.bonnet@ens-lyon.fr}{edouard.bonnet@ens-lyon.fr}\\
	    \and
	    Hung Le\thanks{Hung Le was supported by the NSF CAREER award CCF-2237288 and NSF grant CCF-2517033.}\\[0.05cm]
        \normalsize\af{University of Massachusetts Amherst}\\
        \normalsize\href{mailto:hungle@cs.umass.edu}{hungle@cs.umass.edu}\\
        	    \and
	    Marcin Pilipczuk\thanks{The work of Marcin Pilipczuk on this manuscript was supported by the project BUKA
  that has received funding from the European Research Council (ERC), grant agreement No.~101126229.}\\[0.05cm]
 		\normalsize\af{University of Warsaw}\\
 		\normalsize\href{mailto:marcin.pilipczuk@mimuw.edu.pl}{marcin.pilipczuk@mimuw.edu.pl}\\
	    \and
	    Micha\l{} Pilipczuk\thanks{The work of Micha\l{} Pilipczuk on this manuscript was supported by the project BOBR
	    	that has received funding from the European Research Council (ERC) under the European Union’s Horizon
	    	2020 research and innovation programme (grant agreement No. 948057).}\\[0.05cm]
 		\normalsize\af{University of Warsaw}\\
 		\normalsize\href{mailto:michal.pilipczuk@mimuw.edu.pl}{michal.pilipczuk@mimuw.edu.pl}\\
 }

\begin{document}
\title{Coarse Balanced Separators in Fat-Minor-Free Graphs}

\date{\vspace{-1.5cm}}

\maketitle
 \begin{textblock}{20}(-1.75, 7.3)
 \includegraphics[width=40px]{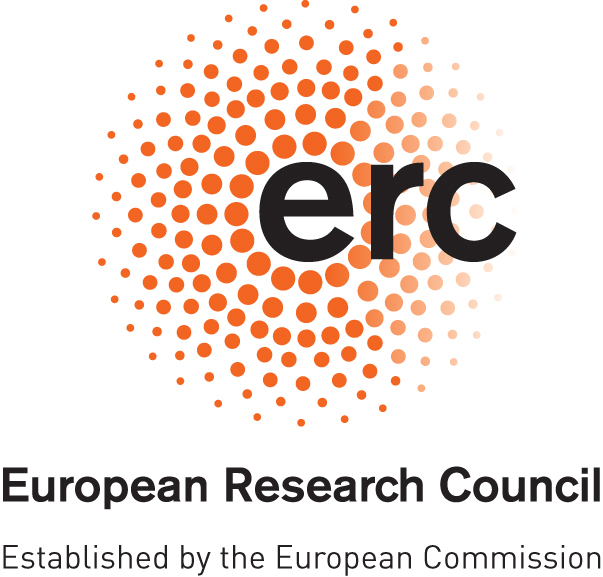}%
 \end{textblock}
 \begin{textblock}{20}(-1.75, 8.3)
 \includegraphics[width=40px]{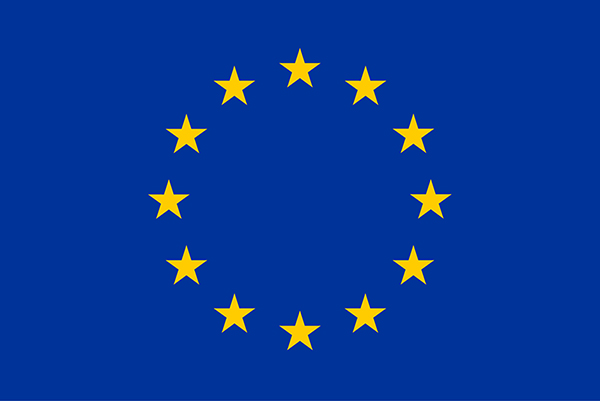}%
 \end{textblock}

\input{abstract}

\input{intro}

\input{prelims}
\input{fatproof}

\input{inducedproof}
\input{conjectures}

\paragraph*{Acknowledgements.} This work was initiated at Dagstuhl Seminar 25212: Metric Sketching and Dynamic Algorithms for Geometric and Topological Graphs. We thank the organizers and the other participants for creating a wonderfully friendly and productive atmosphere.

\bibliographystyle{plain}
\bibliography{ref}

\end{document}

%% file: abstract.tex
\begin{abstract}
  \emph{Fat minors} are a coarse analogue of graph minors where the subgraphs modeling vertices and edges of the embedded graph are required to be distant from each other, instead of just being disjoint. In this paper, we give a coarse analogue of the classic theorem that an $n$-vertex graph excluding a fixed minor admits a balanced separator of size $\Oh(\sqrt{n})$. Specifically, we prove that for every integer~$d$, real $\eps>0$, and graph~$H$, there exist constants $c$ and $r$ such that every $n$-vertex graph $G$ excluding $H$ as a~$d$-fat minor admits a set $S \subseteq V(G)$ that is a balanced separator of $G$ and can be covered by $c n^{\frac{1}{2}+\eps}$ balls of radius $r$ in $G$. Our proof also works in the weighted setting where the balance of the separator is measured with respect to any weight function on the vertices, and is effective: we obtain a~randomized polynomial-time algorithm to compute either such a balanced separator, or a $d$-fat model of $H$ in $G$.
\end{abstract}

  

%% file: intro.tex
\section{Introduction}\label{sec:intro}

 \EMPH{Coarse graph theory} is a relatively young direction in structural graph theory, whose aim is to describe the structure in graphs understood as metric spaces. As suggested in one of the foundational works of Georgakopoulos and Papasoglu~\cite{GeorgakopoulosP25}, the hope is that many classic notions and results of structural graph theory, particularly the theory of graph minors, could be lifted to suitable coarse analogues. The usual way of obtaining such a coarse analogue is to replace any constraints of disjointness or intersection with the requirement that the relevant objects are far or close to each other, respectively. A good example of such a lift is the coarse analogue of the notion of a minor, called a \EMPH{fat minor}. Formally, for graphs $G$ and $H$,
 we say that a~graph $G$ contains $H$ as a {\em{$d$-fat minor}}, for a parameter $d\in \N$, if there are connected subgraphs $\{B_u\colon u\in V(H)\}$ and $\{B_e\colon e\in E(H)\}$ of $G$~such~that
\begin{itemize}
\item whenever $u$ is an endpoint of $e$, the subgraphs $B_u$ and $B_e$ intersect; and
\item except for the above, all the subgraphs in $\{B_f\colon f\in V(H)\cup E(H)\}$ are at distance at least $d$ from each other in $G$.
\end{itemize}
The collection of subgraphs $\{B_u\colon u\in V(H)\}$ and $\{B_e\colon e\in E(H)\}$ as above is called a~\EMPH{$d$-fat model} of $H$ in~$G$.
Thus, the notion of a fat minor requires the constituent parts of the model to be far from each other, rather than just disjoint as in the classic notion of a minor.

In the light of all the deep advances of the classic theory of graph minors, the following question becomes natural: Supposing $G$ excludes a fixed graph $H$ as a $d$-fat minor,  what can be said about the metric structure of $G$? Georgakopoulos and Papasoglu~\cite{GeorgakopoulosP25} conjectured that such graphs are \EMPH{quasi-isometric} to graphs excluding $H$ as a classic minor. Roughly speaking, this would mean that the metric space of a graph excluding $H$ as a $d$-fat minor can be mapped to a metric space of some $H$-minor-free graph~$G'$ so that the mapping preserves the distances up to a constant multiplicative and additive factor (both depending on $d$ and $H$). If this were true, then the wealth of known properties of $H$-minor-free graphs could be pulled through the quasi-isometry to their coarse analogues that would apply to graphs excluding $H$ as a $d$-fat~minor.

Unfortunately, as shown by Davies, Hickingbotham, Illingworth, and McCarty~\cite{DaviesHIM24}, the conjecture of Georgakopoulos and Papasoglu turns out to be false; see also the newer works~\cite{GridCounter,SmallCounter} for simpler and stronger counterexamples. However, many of its potential consequences --- structural properties of graphs excluding a fat minor --- can still be true, at least to some extent. In this work we identify one such property: the existence of balanced~separators of sublinear size.

To measure the balance of a separator, it will be convenient to speak about vertex-weighted graphs. A \EMPH{weighted graph} is a graph $G$ equipped with a vertex weight function $\wei_G\colon V(G)\to \Rp$. (We often omit the subscript if $G$ is clear from the context.) A \EMPH{balanced separator} of $G$ is a set $S\subseteq V(G)$ such that for every connected component $C$ of $G-S$, we have $\wei(C)\leq \wei(G)/2$. Here, if $X$ is a~subgraph or a~vertex subset of $G$, then $\wei(X)$ denotes the total weight of the vertices featured in $X$.

It is well-known that $H$-minor-free graphs admit balanced separators of square root size.

\begin{theorem}[\cite{alon1990separator}]\label{thm:LT}
 For every graph $H$, every $n$-vertex weighted graph that excludes $H$ as a~minor has a~balanced separator of size $\Oh_H(\sqrt{n})$.\footnote{For a tuple of parameters $\bar p$, the $\Oh_{\bar p}(\cdot)$ notation hides factors that may depend on $\bar p$.}
\end{theorem}

Note that for planar graphs (which coincide with $\{K_5,K_{3,3}\}$-minor-free graphs), this is exactly the Planar Separator Theorem of Lipton and Tarjan~\cite{lipton1979separator}. In general, the existence of square-root-sized balanced separators is a fundamental reason behind countless combinatorial and algorithmic results applicable to planar and $H$-minor-free graphs.

A natural coarse analogue of \cref{thm:LT} would be the following statement, which we pose as a~conjecture.
Here, for a graph $G$ and a vertex subset $S\subseteq V(G)$, we say that $S$ is \EMPH{$(k,r)$-coverable} if $S$ can be covered by $k$ balls of radius $r$ in $G$; or in other words, there exists $X\subseteq V(G)$ with $|X|\leq k$ such that every vertex of $S$ is at distance at most $r$ from some vertex of $X$.

\begin{conjecture}\label{conj:fatminor-sep}
For every graph $H$ and $d\in \N$, there exist constants $c,r\in \N$ such that every $n$-vertex weighted graph $G$ that excludes $H$ as a~$d$-fat minor has a~$(c\sqrt{n},r)$-coverable balanced separator. 
\end{conjecture}




\paragraph*{Our results.}
Our main result is a weaker form of \cref{conj:fatminor-sep}, where we lose a multiplicative factor $n^\eps$ in the number of balls covering the separator, for any fixed $\eps>0$. In the formulation below we make the dependence on $H$, $d$, and $\eps$ explicit, as it is actually not too terrible.

\begin{theorem}\label{thm:1/2-bound}
For every graph $H$, integer $d\in \N$, and real $\eps>0$, the following holds: Every $n$-vertex weighted graph that excludes $H$ as a $d$-fat minor has a balanced separator that is $(\Oh(\|H\|^2\cdot n^{1/2+\eps}),\Oh(d/\eps))$-coverable, where $\|H\|=|V(H)|+|E(H)|$.

Furthermore, there is a randomized polynomial-time algorithm that, given a weighted graph $G$, either returns such a balanced separator in $G$ or finds a $d$-fat minor model of $H$ in $G$.
\end{theorem}

We remark that while several previous works studied the structure of $H$-fat-minor-free graphs for specific graphs $H$, to the best of our knowledge, \cref{thm:1/2-bound} offers the first non-trivial structural property of $H$-fat-minor-free graphs proved for an \emph{arbitrary} $H$.
The algorithm of~\cref{thm:1/2-bound} can actually report the centers of the balls covering the separator (rather than the separator itself); alternatively, the centers can be obtained greedily from the separator at the unconsequential expense of doubling the radius.
We also note that by an observation of Abrishami et al.~\cite{CoarseBalSeps}, \cref{thm:1/2-bound} implies that an $n$-vertex graph excluding $H$ as a $d$-fat minor has \EMPH{coarse treewidth} $(\Oh(\|H\|^2\cdot n^{1/2+\eps}),\Oh(d/\eps))$, for any fixed $\eps>0$; we refer to~\cite{CoarseBalSeps} for relevant definitions.



As a side result, we also prove a variant for excluding induced minors in which the balanced separator can be covered by balls of radius~$1$. 
Such balanced separators are often called \EMPH{dominated balanced separators}, and they seem to play an important role in the emerging structure theory for induced minors. Recall here that $H$ is an \EMPH{induced minor} of $G$ if there are pairwise disjoint connected subgraphs $\{B_u\}_{u\in V(H)}$ in $G$ such that for any pair of distinct vertices $u,v\in V(H)$, we have that $u$ and $v$ are adjacent in $H$ if and only if $B_u$ and $B_v$ are adjacent in $G$.

\begin{theorem}\label{thm:star-separator}
  For every graph $H$, every $n$-vertex weighted graph that is $H$-induced-minor-free has an $(\Oh_H(n^{2/3} \log^{4/3} n),1)$-coverable balanced separator.
\end{theorem}

We anticipate that the bound provided by \cref{thm:star-separator} is not tight and the constant $2/3$ in the exponent can be improved to $1/2$. Indeed, we note that this can be done when $H$ is a complete bipartite graph, at the cost of increasing the radius to $4$.

\begin{theorem}\label{thm:induced-biclique}
  For every $t\in \N$, every $n$-vertex weighted graph that is $K_{t,t}$-induced-minor-free has an $(\Oh_t(n^{1/2} \log n),4)$-coverable balanced separator.
\end{theorem}


\paragraph*{Our techniques.} Our proof of \cref{thm:1/2-bound} is inspired by the following result of Korhonen and Lokshtanov~\cite{KL24} about the existence of balanced separators in induced-minor-free graphs.

\begin{theorem}[{\cite[Theorem~1.1]{KL24}}]\label{thm:indminor-sep}
  For every graph $H$, every $m$-edge weighted graph $G$ that is $H$-induced-minor-free has a balanced separator of size at most $\Oh_H(\sqrt{m}\log m)$.
\end{theorem}

We note that in \cite{KL24}, this result is stated for unweighted graphs, but the proof can be easily lifted to the weighted setting. (Alternatively, the weighted result can be obtained from the unweighted one using a theorem of Dvo\v{r}\'ak and Norin~\cite[Theorem~1]{DvorakN19}.) Also, the formulation in \cite{KL24} postulates a stronger bound of $\Oh_H(\sqrt{m})$, but the proof is partly based on a faulty lemma from an earlier work of Lee~\cite{Lee17}. Without reliance on this lemma, the proof goes through with an extra $\log m$~factor.

The idea of Korhonen and Lokshtanov in the proof of \cref{thm:indminor-sep} is to exploit the classic result of Leighton and Rao~\cite{Leighton99} on the gap between product multicommodity flows and balanced separators. Roughly speaking, from this result it follows that if a weighted graph $G$ does not admit a balanced separator of small size, then in $G$ there is a multicommodity flow $\lambda$ that pushes, for every pair of vertices $u,v$, a significant amount of flow from $u$ to $v$, and at the same time every vertex of~$G$ carries only a~relatively small amount of flow. Such a~flow $\lambda$ can be understood as a well-distributed system of paths in $G$. Korhonen and Lokshtanov show that by sampling paths from this system, one can find, with positive probability, an induced minor model of the supposedly excluded graph $H$, leading to a contradiction. 

In our proof of \cref{thm:1/2-bound}, we show that assuming the graph is suitably sparse, the same strategy can be used to expose a fat minor model of $H$, not only an induced one. The precise notion of sparseness has to be chosen carefully. We achieve it using the following clustering result of Filtser~\cite{Filtser24}: for every graph $G$ and $\eps>0$, one can partition the vertex set of $G$ into subsets inducing subgraphs of diameter $\Oh(1/\eps)$, called \EMPH{clusters}, so that every ball of radius $2$ in $G$ intersects only $\Oh(n^\eps)$ clusters. The Leighton--Rao argument is then applied to the graph $G'$ obtained from $G$ by contracting every cluster to a single vertex. Thus, a balanced separator of $G'$ of size $\Oh_H(n^{1/2+\eps})$ can be naturally lifted to a balanced separator of $G$ that can be covered by $\Oh_H(n^{1/2+\eps})$ balls of bounded radius. 

The proof of \cref{thm:star-separator} is much less involved. We first apply a simple clustering scheme that vertex-partitions the graph into connected subgraphs, each dominated by a~single vertex, so that contracting the subgraphs yields a graph with $\Oh_H(n^{4/3}\log^{2/3} n)$ edges. Then we apply \cref{thm:indminor-sep} as a black-box in the contracted graph, and again lift the obtained balanced separator to the original graph. \cref{thm:induced-biclique} follows by the same approach, but we apply an off-the-shelf clustering result for $K_{t,t}$-induced-minor-free graphs due to Chudnovsky, Codsi, Ajaykrishnan, and Lokshtanov~\cite{logMenger}.

%% file: prelims.tex
\section{Preliminaries}\label{sec:prelims}

For a nonnegative integer $k$, we denote $[k]\coloneqq \{1,\ldots,k\}$.

\paragraph*{Graphs.} All graphs considered in this paper are finite, simple (with no parallel edges or loops), and undirected, unless explicitly stated. We use standard graph notation. In particular, for a graph $G$, we write $V(G)$ for the vertex set of $G$, $E(G)$ for the edge set of $G$, $\dist_G(\cdot,\cdot)$ for the distance metric in $G$ (mapping every pair of vertices $u,v$ to the length, i.e., number of edges, on a shortest path between $u$ and $v$), and $\diam(G)$ for the \EMPH{diameter} of $G$, which is the largest distance between a pair of vertices of~$G$. When $X,Y$ are subgraphs or subsets of vertices of $G$, we define $\dist_G(X,Y)$ as the smallest distance in $G$ between a vertex of $X$ and a vertex of $Y$. For a graph $G$ and a vertex or edge $x$, we write $x\in G$ to signify that $x$ belongs to the vertex set or the edge set of $G$, respectively. For a pair of graphs $G,G'$, we write $G\cup G'$ for the graph $(V(G)\cup V(G'),E(G)\cup E(G'))$.

For $r\in \N$ and a graph $G$, the \EMPH{ball} of radius $r$ with center $u\in V(G)$ is the set $B_G(u,r)\coloneqq \{v\in V(G)\mid \dist_G(u,v)\leq r\}$. The \EMPH{$r$th power} of $G$ is the graph $G^r$ on the same vertex set as $G$ where two vertices $u,v$ are adjacent if and only if they are at distance at most $r$ in $G$.

For a (possibly directed) graph $G$ and a pair of distinct vertices $u,v\in V(G)$, we let $\Pp_{u,v}(G)$ be the set of all paths from $u$ to $v$ in $G$. Then $\Pp(G)\coloneqq \bigcup_{u,v\in V(G), u\neq v} \Pp_{u,v}(G)$ is the set of all paths in $G$ on at~least two vertices.
If $G$ is directed, then we mean the paths as directed.

Weighted graphs and balanced separators have already been defined in \cref{sec:intro}. Note that we use vertex weights only to measure the balance of separators. In particular, there are no weights on the edges and the lengths of paths are measured by the number of edges. When taking (induced) subgraphs in the context of weighted graphs, we naturally restrict the weight function to the remaining vertices.


A \EMPH{connected partition} of a graph $G$ is a partition $\Xx$ of the vertex set of $G$ such that for each set $X \in \Xx$, the graph $G[X]$ is connected. Sets $X\in \Xx$ are called the \EMPH{clusters} of~$\Xx$. The \EMPH{strong diameter} of $\Xx$ is $\max_{X\in \Xx}\diam(G[X])$, and the \EMPH{weak diameter} of $\Xx$ is $\max_{X\in \Xx}\max_{x,y\in X}\dist_G(x,y)$. Finally, the \EMPH{quotient graph} $G/\Xx$ is the graph with vertex set $\Xx$ where two clusters $X,Y\in \Xx$ are adjacent if in $G$ there is an edge with one endpoint in $X$ and the other in $Y$. In other words, $G/\Xx$ is obtained from $G$ by contracting every cluster into a single vertex and removing parallel edges and self-loops.

We will later use the following lemma, a~consequence of the work of Filtser.
\begin{lemma}[Consequence of {\cite[Theorem~4]{Filtser24}}]\label{lem:partition}
  There is a~randomized polynomial-time algorithm that, given a graph $G$ and a real $\eps > 0$, outputs a connected~partition $\Xx$ of~$G$ of strong diameter at~most~$\frac{32}{\eps}$ such that for every vertex $u \in V(G)$, the ball $B_G(u,2)$ intersects $\Oh(n^\eps)$ clusters of~$\Xx$.
\end{lemma}
\begin{proof}
In~\cite[Theorem 4]{Filtser24}, set $\Delta \coloneqq \frac{32}{\eps}$ and $\alpha \coloneqq \frac{2}{\eps}$.
\end{proof}

\paragraph*{Minors.} The notions of minors, fat minors, and induced minors were also already introduced in \cref{sec:intro}. Similarly to the work of Korhonen and Lokshtanov~\cite{KL24} on induced minors, it will be useful to work with the following relaxed notion of fat models.
A \EMPH{crude $d$-fat model} of a graph $H$ in a~graph~$G$ is a~pair of mappings $(\phi, \pi)$ such that $\phi$ maps vertices of $H$ to vertices of $G$, $\pi$ maps edges of $H$ to paths in $G$, and the following conditions hold:
\begin{itemize}
    \item For every edge $uv \in E(H)$, $\pi(uv)$ is a $\phi(u)$--$\phi(v)$ path in~$G$.
    \item For every pair of edges $uv, u'v' \in E(H)$ with $|\{u,v,u',v'\}|=4$, we have \mbox{$\dist_G(\pi(uv), \pi(u'v')) \geq d$.} 
\end{itemize}
Note that we do not require the vertices $\phi(u)$, for $u \in V(H)$, to be far from each other, or even distinct. The only distance constraints are between paths modeling pairs of edges that do not share endpoints.

The following simple lemma, directly lifted from an analogous statement in \cite{KL24}, explains that a $d$-fat model of a graph $H$ can be easily extracted from a~crude $d$-fat model of the \EMPH{$2$-subdivision} $\Ddot{H}$ of $H$ --- the graph obtained from $H$ by replacing every edge by a path of length $3$.


\begin{lemma}\label{lem:d-fat-almost}
 Let $H$ be a graph with no isolated vertices.
 Suppose $G$ contains a crude $d$-fat model of~$\Ddot{H}$.
 Then $G$ also contains a $d$-fat model of $H$.
\end{lemma}
\begin{proof}
  Let $(\phi,\pi)$ be a crude $d$-fat model of $\Ddot{H}$ in $G$. 
  For every vertex $u \in V(H)$, let $\Ddot{E}_u$ be the set of all edges incident to $u$ in $\Ddot{H}$.
  We set \[B_u \coloneqq \bigcup_{e\in \Ddot{E}_u}\pi(e).\]
  For every edge $e \in E(H)$, let $\Ddot{e} \in E(\Ddot{H})$ be the middle edge of the $3$-edge-path into which~$e$ got subdivided in $\Ddot{H}$.
  We set \[B_e \coloneqq \pi(\Ddot{e}).\]
  We now show that $(\{B_u\}_{u\in V(H)}, \{B_e\}_{e\in E(H)})$ is a~$d$-fat model~of $H$ in~$G$.
  First, note that every subgraph $B_u$ and $B_e$ is indeed connected. Second, observe that for every vertex $u$ of $H$ and edge $e$ of $H$ incident to~$u$, the subgraph $B_u$ and $B_e$ intersect, because they both contain $\phi(u')$, where $u'$ is the endpoint of $\Ddot{e}$ that is adjacent to $u$ in $\Ddot{H}$.


  We are left with verifying the distance condition. 
  Let $X,Y$ be two members of the collection of subgraphs $\{B_u\}_{u\in V(H)}\cup \{B_e\}_{e\in E(H)}$.
  We have to consider three cases.
    
  \paragraph*{Case 1:} Suppose $X = B_{u}$ and $Y = B_{v}$ for distinct $u,v \in V(H)$.
  Let $x\in  B_{u}$ and $y\in  B_{v}$ be any two vertices.
  Then, by construction, we have $x\in \pi(e_x)$ and $y \in \pi(e_y)$ for some edges $e_x, e_y \in E(\Ddot{H})$ incident to $u$ and $v$, respectively.
  Observe that $e_x$ and $e_y$ do not share any endpoint, and therefore we have $\dist_G(\pi(e_x),\pi(e_y))\geq d$. Thus $\dist_G(x,y)\geq d$, and hence also $\dist_G(B_u,B_v) \geq d$. 

  \paragraph*{Case 2:} Suppose $X = B_{u}$ and $Y = B_{e}$ for some edge $e$ that is not incident to $u$.
  Then for every edge $f \in \Ddot{E}_u$, $f$ and $\Ddot{e}$ do not share any endpoint.
  It follows that $\dist_G(\pi(f),\pi(\Ddot{e})) \geq d$, and hence also $\dist_G(B_u,B_e) \geq d$. 

  \paragraph*{Case 3:} Suppose $X = B_{e}$ and $Y = B_{f}$ for some distinct edges $e$ and $f$.
  Then $\Ddot{e}$ and $\Ddot{f}$ do not share any endpoint and therefore, $\dist_G(B_e,B_f)=\dist_G(\pi(\Ddot{e}),\pi(\Ddot{f}))\geq d$. 
\end{proof}

We note that the construction from the proof of~\cref{lem:d-fat-almost} can be easily turned into a polynomial-time algorithm for converting a crude $d$-fat model of~$\Ddot{H}$ into a~$d$-fat model of~$H$.



\paragraph*{Concurrent flows and balanced separators.} Our proof of \cref{thm:1/2-bound} crucially relies on the duality between multicommodity flows and balanced cuts, proved by Leighton and Rao~\cite{Leighton99}. Let us recall the relevant results.

The work of Leighton and Rao concerns the \EMPH{Product Multicommodity Flow Problem} (PMFP) on directed edge-capacitated graphs.
An instance of (directed) PMFP is a~triple $(D,c,\pi)$ where 
\begin{itemize}
    \item $D$ is a~(directed) graph,
    \item $c\colon E(D) \to \Rp$ is an edge-capacity function, and
    \item $\pi\colon V(D) \to \Rp$ is a vertex demand function.
\end{itemize}  
For every pair of distinct vertices $u,v\in V(D)$,
the \EMPH{demand} from $u$ to~$v$ is $\pi(u)\pi(v)$ units of commodity~$(u,v)$.
The goal is to find a~(multicommodity) flow --- union of flows of commodity $(u,v)$ for every pair of distinct vertices $u,v \in V(D)$ --- that does not exceed the capacity of any edge and routes for all these pairs simultaneously $q \cdot \pi(u)\pi(v)$ units of the commodity $(u,v)$ from $u$ to $v$, for the largest possible value~$q$. Formally, a multicommodity flow of value $q\in \Rp$ is a function $f\colon \Pp(D)\to \Rp$, assigning each path $P\in \Pp(D)$ the amount of flow passing through $P$, so~that
\[\sum_{P\in \Pp(D), e\in P} f(P)\leq c(e)\quad\textrm{for all } e\in E(D)~\textrm{ and} \sum_{P\in \Pp_{u,v}(D)} f(P)=q \cdot \pi(u) \pi(v) \quad\textrm{for all } u,v\in V(D), u\neq v.\] 
On the other hand of the duality we have balanced cuts, where the balance is measured as follows: For a~partition $(U,\overline U)$ of $V(D)$ such that $\pi(U) \pi(\overline U) > 0$ (with $\pi(X) := \sum_{v \in X} \pi(v)$), we define its \EMPH{weighted ratio cost} as
  \[\frac{\sum\limits_{e \in E_{D}(U,\overline U)} c(e)}{\pi(U) \cdot \pi(\overline U)},\]
  where $E_D(U,\overline{U})$ denotes the set of edges of $D$ with tail in $U$ and head in $\overline{U}$.
  With these definitions in place, we can recall the main result of Leighton and Rao.

\begin{theorem}[{\cite[Theorem~17]{Leighton99}}]\label{thm:LR}
 There is a randomized polynomial-time algorithm that given an instance $(D,c,\pi)$ of the Directed Product Multicommodity Flow Problem and a real $q\in \Rp$, outputs either
 \begin{enumerate}[label=(\arabic*)]
     \item a multicommodity flow of value $q$ in $(D,c,\pi)$, or
     \item a partition $(U,\overline{U})$ of $V(D)$ of weighted ratio cost $\Oh(q \log p)$, where $p \coloneqq |\{v \in V(D) \colon \pi(v) > 0\}|$.
 \end{enumerate}
\end{theorem}

We remark that the formulation of \cite[Theorem~17]{Leighton99} is non-algorithmic, but the proof is effective and can be turned into a randomized polynomial-time algorithm; see the remarks after \cite[Theorem 7]{Leighton99}.

Similarly as in \cite{KL24}, for our purposes we need to translate the duality provided by \cref{thm:LR} to the setting of vertex separators, vertex-capacitated flows, and (vertex-)weighted undirected graphs.
The translation uses a~standard, simple reduction scheme: splitting every vertex $u$ into two. However, we need to first introduce the relevant terminology. 

The role of multicommodity flows will be played by \EMPH{concurrent flows}.
For a weighted graph $G$, a~\EMPH{concurrent flow} in $G$ is a function $\lambda\colon \Pp(G) \rightarrow \Rp$ that assigns each path $P\in \Pp(G)$ the amount of flow passing through $P$, so that for every pair of vertices $(u,v) \in V(G) \times V(G)$ with $u \neq v$, we have \[\sum_{P \in\Pp_{u,v}(G)}\lambda(P) = \wei(u) \wei(v).\]
By $\gamma_{G,\lambda}(v)$ we denote the total amount of flow passing through $v$, that is,
\begin{equation}
    \gamma_{G,\lambda}(v)  \coloneqq \sum_{P \in \Pp(G), v \in P} \lambda(P)
\end{equation}
The \EMPH{congestion} $\gamma$ of the concurrent flow $\lambda$ is $\max_{v\in V(G)}\gamma_{G,\lambda}(v)$.

On the other hand, the role of balanced cuts will be played by balanced separations. Recall that a~\EMPH{separation} in a (weighted) graph $G$ is a pair of vertex subsets $(A,B)$ such that $A\cup B=V(G)$ and there is no edge with one endpoint in $A\setminus B$ and the other in $B\setminus A$. Then $A\cap B$ is the \EMPH{separator} of the separation $(A,B)$.
Assuming $G$ is weighted, we define the \EMPH{sparsity} of $(A,B)$ as
\[\alpha_{G}(A,B) \coloneqq \frac{|A\cap B|}{\wei(A) \cdot \wei(B)}.\]
With all the definitions in place, we can prove the desired variant of the result of Leighton and~Rao.


\begin{lemma}[undirected vertex-weighted variant of~{\cite[Theorem 17]{Leighton99}}]\label{lem:LR-reformulation}
  There is a~randomized polynomial-time algorithm that, given a weighted $n$-vertex graph $G$ and a real $\gamma > 0$, outputs either:
  \begin{enumerate}[label=(\arabic*)]
    \item \label{it:LR-1} a concurrent flow in $G$ with congestion at~most~$\gamma$, or
    \item \label{it:LR-2} a~separation in $G$ of sparsity $\Oh\left(\frac{\log n}{\gamma}\right)$.
\end{enumerate}
\end{lemma}
\begin{proof}
  We perform the standard trick of splitting every vertex into two.
  Let $D$ be the~directed graph with edge capacities $c\colon E(D)\to \Rp$ obtained from $G$ as follows: 
  \begin{itemize}
      \item for every $u\in V(G)$, add to $D$ two vertices $u_{\text{in}}$ and $u_{\text{out}}$ and connect them by an edge $(u_{\text{in}},u_{\text{out}})$ of capacity $1$ and an edge  $(u_{\text{out}},u_{\text{in}})$ of capacity $\infty$; and
      \item for every edge $uv\in E(G)$, add to $D$ edges $(u_{\text{out}},v_{\text{in}})$ and $(v_{\text{out}},u_{\text{in}})$ of capacity $\infty$.
  \end{itemize}
  Edges of capacity $\infty$ can be avoided by using a large constant, e.g. $|V(G)|^2\cdot \wei(G)+1$, instead.
  Finally, for every $u\in V(G)$, we set
  \[\pi(u_{\text{in}}) \coloneqq \wei(u)\qquad \text{and} \qquad\pi(u_{\text{out}}) \coloneqq 0.\]
  Thus, $(D,c,\pi)$ is an instance of the Directed Product Multicommodity Flow Problem. Let us verify that multicommodity flows in this instance are in correspondence to concurrent flows in $G$.

  
  \begin{claim}\label{clm:multic-conc}
    For any real $q>0$,
    $(D,c,\pi)$ has a~multicommodity flow of value $q$ if and only if $G$ has a~concurrent flow of~congestion~$1/q$.
  \end{claim}
  \begin{proof}
  First, suppose there is a~multicommodity flow $f$ in~$(D,c,\pi)$ that, for every pair of distinct vertices $u,v \in V(G)$, sends $q \cdot \pi(u_{\text{in}})\pi(v_{\text{in}})$ units of flow from $u_{\text{in}}$ to $v_{\text{in}}$. Note that there is a natural map from $\Pp_{u_{\text{in}},v_{\text{in}}}(D)$ to $\Pp_{u,v}(G)$ defined by contracting every vertex pair $w_{\text{in}},w_{\text{out}}$ back to $w$. Hence, we may map $f$ to a flow $\lambda_0$ in $G$ that sends $q\cdot \wei(u)\wei(v)$ units of flow from $u$ to $v$, for every pair of distinct vertices $u,v$, and has congestion $1$ (this corresponds to the edges of the form $(w_{\text{in}},w_{\text{out}})$ having capacity~$1$). By rescaling $\lambda_0$ by a multiplicative factor of $1/q$, we obtain a concurrent flow $\lambda$ in $G$ with congestion $1/q$.
  

  Conversely, suppose $G$ has a~concurrent flow $\lambda$ of~congestion~$1/q$. Similarly as above, we may map $\lambda$ to a multicommodity flow $f_0$ in $(D,c,\pi)$ that sends $\wei(u)\wei(v)=\pi(u_{\text{in}})\pi(v_{\text{in}})$ units of flow from 
  $u_{\text{in}}$ to $v_{\text{in}}$, for every pair of distinct vertices $u,v$, and uses capacity $1/q$ at every edge of the form $(w_{\text{in}},w_{\text{out}})$. By rescaling $f_0$  by a~multiplicative factor of~$q$, we obtain a multicommodity flow $f$ of value $q$ that respects all the edge capacities.
  \end{proof}

  We apply the algorithm of \cref{thm:LR} to $(D,c,\pi)$ with $q=1/\gamma$.
 If this application yields a~multicommodity flow of value $1/\gamma$, then by~\cref{clm:multic-conc} we may obtain a~concurrent flow in~$G$ with congestion~$\gamma$, which is outcome~\ref{it:LR-1}. So we may assume that we have computed a~partition $(U,\overline U)$ of $V(D)$ with weighted ratio cost bounded by $\Oh(\log n/\gamma)$.



  Observe that every edge in $E_D(U,\overline{U})$ must be of the form $(u_{\text{in}},u_{\text{out}})$ for some $u \in V(G)$, because the other edges have infinite capacity, which would render the weighted ratio cost of $(U,\overline U)$ infinite (or too large).
  Let $S \subseteq V(G)$ be the set comprised of all vertices $u$ such that $(u_{\text{in}},u_{\text{out}})\in E_D(U,\overline{U})$.
  Then, \[\sum_{e \in E_{D}(U,\overline U)} c(e) = |E_{D}(U,\overline U)| = |S|.\]
  Let $A \subseteq V(G)$ be the set of vertices $u$ such that $u_{\text{in}}\in U$, and $B\subseteq V(G)$ be the set of vertices $u$ such that $u_{\text{out}}\in \overline{U}$; thus $S=A\cap B$. Note that due to the existence of an edge $(u_{\text{out}},u_{\text{in}})$ of capacity $\infty$, for each vertex $u\in V(G)$ we must have $u\in A$ or $u\in B$. Moreover, since every edge $uv$ of $G$ gives rise to edges $(u_{\text{out}},v_{\text{in}})$ and $(v_{\text{out}},u_{\text{in}})$ of capacity $\infty$, it cannot happen that $u\in A\setminus B$ (implying $u_{\text{out}}\in U$) and $v\in B\setminus A$ (implying $v_{\text{in}}\in \overline{U}$) simultaneously. We conclude that $(A,B)$ is a separation of $G$, with separator $S=A\cap B$.
  
  Finally, observe that the sparsity of $(A,B)$ is \[\alpha_{G}(A,B) = \frac{|S|}{\wei(A) \cdot \wei(B)}= \frac{\sum\limits_{e \in E_{D}(U,\overline U)} c(e)}{\wei(A) \cdot \wei(B)}.\]
  Note that $\pi(U) = \wei(A)$ and $\pi(\overline U) = \wei(B\setminus A)\leq \wei(B)$.
  Thus we have \[\alpha_{G}(A,B) \leqslant \frac{\sum\limits_{e \in E_{D}(U,\overline U)} c(e)}{\pi(U) \cdot \pi(\overline U)},\]
  which in turn is bounded by $\Oh(\log n/\gamma)$. So the separation $(A,B)$ constitutes a valid outcome~\ref{it:LR-2}.
\end{proof}

We now transform, by repeated applications, the second outcome of~\cref{lem:LR-reformulation} into a~balanced separator of the graph. This is analogous to~\cite[Lemma~4.1]{KL24}, but we repeat the argument here for the convenience of the reader, and also because the presentation of \cite{KL24} claims a slightly stronger result assuming a faulty lemma from~\cite{Lee17}.
 
\begin{lemma}\label{lm:sep-vs-flow-weighted}
  There is a~randomized polynomial-time algorithm that, given a weighted $n$-vertex graph $G$ with $W\coloneqq \wei(G)$ and a real $\gamma > 0$, outputs either:
  \begin{enumerate}[label=(\arabic*)]
    \item \label{it:sep-vs-flow-1} an induced subgraph $G'$ of $G$ such that $\wei(G')\geq W/2$, a concurrent flow in $G'$ with congestion at~most $\gamma$, and a~subset $S \subseteq V(G)$ of size $\Oh\left(\frac{W^2\log n}{\gamma}\right)$ such that $S\cap V(G')=\emptyset$ and every connected component of $G-S$ is either disjoint from $G'$ or contained in $G'$; or
    \item \label{it:sep-vs-flow-2} a balanced separator of $G$ of size $\Oh\left(\frac{W^2\log n}{\gamma}\right)$.
\end{enumerate}
\end{lemma}
\begin{proof}
  We build a~sequence $G_1,G_2,G_3,\ldots$ of induced subgraphs of $G$ as follows. Note here that each $G_i$, as an induced subgraph of $G$, inherits the weight function $\wei\coloneqq \wei_G$ from $G$.

  Set $G_1 \coloneqq G$ and $i \leftarrow 1$.  
  Perform the following loop: While $\wei(G_i) \geq W/2$, apply \cref{lem:LR-reformulation} to~$G_i$ with congestion parameter $\gamma$.
  If this application yields outcome~\ref{it:LR-1}, then we break out of the while loop.
  Otherwise, we obtain a~separation $(A_i,B_i)$ in~$G_i$ with sparsity at~most~$\frac{c \log n}{\gamma}$, for some universal constant~$c$.
  By swapping $A_i$ and $B_i$ if necessary, we assume that $\wei(A_i) \geq \wei(B_i)$.
  We set $G_{i+1} \coloneqq G_i[A_i\setminus B_i]$ and $i \leftarrow i+1$, and proceed to the next iteration.
  This ends the body of the~loop.

  When we exit the while loop, we have built separations $(A_1,B_1), \ldots, (A_\ell,B_\ell)$ of induced subgraphs $G_1\supseteq \ldots \supseteq G_\ell$, respectively. For $i\in [\ell]$, denote $S_i\coloneqq A_i\cap B_i$.
  We claim that if we have not broken out of the while loop (upon finding a~concurrent flow), then $S \coloneqq \bigcup_{i \in [\ell]} S_i$ is a~balanced separator of~$G$.
  Note that by construction, $V(G)\setminus S$ is partitioned into $B_1\setminus A_1, B_2\setminus A_2, \ldots, B_\ell\setminus A_\ell$, and $A_\ell\setminus B_\ell$, with no edge between two distinct parts.
  For each $i \in [\ell]$, we have $\wei(B_i\setminus A_i) \leq W/2$, because $\wei(B_i) \leq \wei(A_i)$.
  By the condition of exiting the loop, we have $\wei(A_\ell\setminus B_\ell) \leqslant  W/2$. So $S$ is indeed a balanced separator of~$G$. 

  We are left with arguing that $|S| \leq \Oh\left(\frac{W^2\log n}{\gamma}\right)$.
  For every $i \in [\ell]$ we have
  \[\alpha_{G_i}(A_i,B_i) = \frac{|S_i|}{\wei(A_i) \cdot \wei(B_i)} \leq \frac{c \log n}{\gamma},~~\text{thus}\]
  \[|S| = \sum_{i \in [\ell]} |S_i| \leq \frac{c \log n}{\gamma} \cdot \sum\limits_{i \in [\ell]} \wei(A_i) \cdot \wei(B_i).\]
  Note that the sets $B_1, \ldots, B_\ell$ are pairwise disjoint, and we can upper bound every term $\wei(A_i)$ by the total weight~$W$. Therefore
  \[|S| \leq \frac{c W \log n}{\gamma} \cdot \sum\limits_{i \in [\ell]} \wei(B_i) \leq \frac{c W^2 \log n}{\gamma},\]
  and $S$ constitutes a~valid outcome \eqref{it:sep-vs-flow-2}.

  Now, suppose we have broken out of the while loop because a~concurrent flow $\lambda$ of congestion at most $\gamma$ was found in~$G' \coloneqq G[A_\ell\setminus B_\ell]$. By the invariant of the loop, $G'$ is an induced subgraph of~$G$ such that $\wei(G') \geq W/2$. And by construction, $S$ is disjoint with $V(G')$ and every connected component of $G-S$ is either contained in $G'$ or disjoint with $G'$. Also, by the same calculation as above, we have $|S|\leq \Oh\left(\frac{W^2\log n}{\gamma}\right)$.  
  We conclude that $G'$, $\lambda$, and $S$ satisfy the desired properties of outcome~\eqref{it:sep-vs-flow-1}.
\end{proof}

%% file: fatproof.tex
\section{Proof of~\cref{thm:1/2-bound}}\label{sec:fatproof}

In this section we prove our main result, \cref{thm:1/2-bound}.
By the following result, the statement can be reduced to the $d=3$ case. Recall here that $G^d$ is the $d$th power of $G$ --- the graph obtained from $G$ by making two vertices of $V(G)$ adjacent if and only if they are at distance at most $d$ in $G$.

\begin{theorem}[{\cite[Theorem~3]{DaviesHIM24}}]\label{thm:reduction-to-3}
  If $G$ excludes $H$ as a $d$-fat minor, then $G^d$ excludes $H$ as a 3-fat minor.
\end{theorem}

Indeed, assuming $d > 3$, we apply the \emph{3-fat minor} case of~\cref{thm:1/2-bound} to $G^d$.
This yields either a~$3$-fat model of $H$ in $G^d$, which can be turned into a $d$-fat model of $H$ in $G$ (the proof of \cref{thm:reduction-to-3} yields a polynomial-time algorithm), or a~set $S \subseteq V(G^d)=V(G)$ that is a balanced separator of $G^d$ and is $(\Oh(\|H\|^2\cdot n^{1/2+\eps}),\Oh(1/\eps))$-coverable in $G^d$. Clearly, as $G$ is a subgraph of $G^d$, $S$ is also a balanced separator of $G$. Moreover, as the distances in $G$ are at most $d$ times larger than in $G^d$, it follows that $S$ is $(\Oh(\|H\|^2\cdot n^{1/2+\eps}),\Oh(d/\eps))$-coverable in $G$.


Therefore, we are left with proving the following statement.

\begin{theorem}\label{thm:1/2-bound-3}
For every graph $H$ and real $\eps>0$, every $n$-vertex weighted graph that excludes $H$ as a $3$-fat minor has an $(\Oh(\|H\|^2\cdot n^{1/2+\eps}),\Oh(1/\eps))$-coverable balanced~separator.

Furthermore, there is a randomized polynomial-time algorithm that, given a weighted graph $G$, either returns such a balanced separator in $G$ or finds a $3$-fat model of $H$ in $G$.
\end{theorem}
\begin{proof}
 We focus on showing the algorithmic claim, as it implies the first part of the theorem statement. By adding some edges to $H$ if necessary, we may assume that $H$ has no isolated vertices; this increases $\|H\|$ only by a constant multiplicative factor, and if a~$3$-fat minor model of this supergraph is found, one can effectively restrict the model to one for the original~$H$. Denote $W\coloneqq \wei_G(G)$, for brevity.

 By applying the algorithm of~\cref{lem:partition}, we may compute, in randomized polynomial time, a~connected partition $\Xx$ of $G$ of strong diameter at~most~$32/\eps$ such that for every vertex $v$ of $G$, the ball $B_G(v,2)$ intersects at most $\Oh(n^\eps)$ clusters of $\Xx$. We observe that from this it follows that there are relatively few clusters of $\Xx$ that are close to each other. Formally, we consider the set
 \[\Xi\coloneqq \left\{\,(X,X')\in \Xx\times \Xx\ \mid\ \dist_G(X,X')\leq 2\,\right\}.\]
  
 \begin{claim}\label{clm:few-bad-pairs}
  $|\Xi|\leq \Oh(n^{1+\eps})$. 
 \end{claim}
 \begin{proof}
   Since every vertex of $G$ is at distance at most $2$ from at most $\Oh(n^\eps)$ clusters of $\Xx$, each cluster $X\in \Xx$ is at distance at most $2$ from at most $|X| \cdot \Oh(n^\eps)$ other clusters of $\Xx$.
   Thus, we have
   \[|\Xi|\leq \sum\limits_{X\in \Xx} \left(|X| \cdot \Oh(n^\eps)\right) = \left(\sum\limits_{X\in \Xx} |X|\right) \cdot \Oh(n^\eps)=n \cdot \Oh(n^\eps)=\Oh(n^{1+\eps}).\qedhere\]
 \end{proof}

 Let $c$ be the constant hidden in the $\Oh(\cdot)$ notation of \cref{clm:few-bad-pairs}; that is, we have $|\Xi|\leq c n^{1+\eps}$.
 We let $\wh{G} \coloneqq G/\Xx$ be the quotient graph. Recall that $\wh{G}$ is weighted with the weight function defined as $\wei_{\wh{G}}(X)\coloneqq \wei_G(X)$ for all $X\in \Xx$, so in particular $\wei_{\wh{G}}(\wh{G})=\wei_G(G)=W$.
 
 We apply the algorithm of \cref{lm:sep-vs-flow-weighted} to the (weighted) graph $\wh{G}$ with
 \[\gamma \coloneqq \frac{W^2}{32h\sqrt{c}\cdot n^{\frac{1+\eps}{2}}}\]
 where $h\coloneqq \|\Ddot{H}\|=|V(\Ddot{H})|+|E(\Ddot{H})|$ is the total number of vertices and edges of~$\Ddot{H}$; note that $h \leq 5 \|H\|$.
 If this application yields outcome \ref{it:sep-vs-flow-2}, we have obtained a~balanced separator $\Ss\subseteq \Xx$ of~$\wh{G}$ of size $\Oh(W^2 \log n/\gamma)\leq\Oh(\|H\|\cdot n^{1/2+\eps})$.
 Hence, $\bigcup \Ss$ is a~balanced separator of~$G$ that can be covered by $|\Ss|=\Oh(\|H\|\cdot n^{1/2+\eps})$ balls of radius~$32/\eps$ in $G$.

 We thus assume that the application of \cref{lm:sep-vs-flow-weighted} yields outcome \ref{it:sep-vs-flow-1}. That is, we have obtained an induced subgraph $\wh{G}'$ of~$\wh{G}$ with $\wei_{\wh{G}}(\wh{G}') \geq \wei_{\wh{G}}(\wh{G})/2=W/2$, a concurrent flow $\lambda$ on $\wh{G}'$ with congestion at~most~$\gamma$, and a subset
 $\wh\Ss \subseteq V(\wh{G})=\Xx$ of~size at~most~$\Oh\left(W^2\log n/\gamma\right)=\Oh(\|H\|\cdot n^{1/2+\eps})$ such that every connected component of~$\wh{G}-\wh\Ss$ is disjoint from~$\wh{G}'$ or contained in $\wh{G}'$. Note that every component of $\wh{G}-\wh\Ss$ disjoint from $\wh{G}'$ must have weight at most $W/2$, for $\wh{G}'$ has weight at least $W/2$.

 We now partition the vertex set of $\wh{G}'$ into two sets $\Zz$ and $\Uu$ as follows:
 \begin{itemize}
     \item $\Zz$ comprises all $X\in V(\wh{G}')$ such that $\wei_G(X)\geq \frac{W}{4h^2}$; and
     \item $\Uu$ comprises all the other vertices of $\wh{G}'$, i.e., those $X\in V(\wh{G}')$ that satisfy $\wei_G(X)<\frac{W}{4h^2}$.
 \end{itemize}
 Note that since $\wei_{\wh{G}}(\wh{G}')\leq W$, we have $|\Zz|\leq 4h^2\leq \Oh(\|H\|^2)$.

 We first consider the case when $\wei_{\wh{G}}(\Zz)\geq \wei_{\wh{G}}(\Uu)$. This means that $\Zz$ contains at least half of the weight of $\wh{G}'$, and hence $\wh{\Ss}\cup \Zz$ is a balanced separator of $\wh{G}$. Therefore, $S\coloneqq \bigcup (\wh{\Ss}\cup \Zz)$ is a balanced separator of $G$, and it can be covered by $|\wh{\Ss}|+|\Zz|\leq \Oh(\|H\|^2\cdot n^{\frac{1}{2}+\eps})$ balls of radius $32/\eps$ in $G$.
 

 Hence, from now on we assume that $\wei_{\wh{G}}(\Uu)> \wei_{\wh{G}}(\Zz)$. In particular, this implies that
 \[M\coloneqq \wei_{\wh{G}}(\Uu)\geq \wei_{\wh{G}}(\wh{G}')/2\geq W/4.\]
 
 We now construct a crude $3$-fat model $(\phi\colon V(\Ddot{H}) \to V(\wh{G}),\pi\colon E(\Ddot{H}) \to \Pp(\wh{G}))$ of $\Ddot{H}$ in $\wh{G}$ as follows.
 Let $\Dd$ be the probability distribution over $\Uu$ that assigns probability $\wei_G(X)/M$ to every $X \in \Uu$.
 For every pair $(X,Y)$ of distinct elements of~$\Uu$, let $\Dd_{X,Y}$ be the probability distribution over $\Pp_{X,Y}(\wh{G}')$ that assigns probability \[\frac{\lambda(P)}{\wei_G(X)\wei_G(Y)}~\text{ to every } P \in \Pp_{X,Y}(\wh{G}').\]
 For every $u \in V(\Ddot{H})$, we sample $\phi(u)$ independently from~$\Dd$.
 Then, for every $uv \in E(\Ddot{H})$, if $\phi(u) \neq \phi(v)$, then we sample $\pi(uv)$ independently from~$\Dd_{\phi(u),\phi(v)}$.
 If $\phi(u) = \phi(v)$, then $\pi(uv)$ is instead set to the trivial $\phi(u)$--$\phi(v)$ path.

 We consider the following two events:
 \begin{description}
     \item $I$: The mapping $\phi$ is injective.
     \item $F$: For every pair of edges $uv,u'v'\in E(\Ddot{H})$ with $|\{u,v,u',v'\}|=4$, we have $\dist_{\wh{G}}(\pi(uv),\pi(u'v'))\geq 3$.
 \end{description}
 As usual, by $\overline{I}$ and $\overline{F}$ we denote the complements of those events, respectively.

 We first bound the probability of $\overline{I}$:
 \begin{align}
     \Pr[\overline{I}] & \leq \sum_{u,v\in V(\Ddot{H}), u\neq v} \Pr[\phi(u)=\phi(v)]&& \nonumber\\
            & = \sum_{u,v\in V(\Ddot{H}), u\neq v}\ \sum_{X\in \Uu} \Pr[\phi(u)=X]\cdot \Pr[\phi(v)=X]&& \nonumber\\
            & = \sum_{u,v\in V(\Ddot{H}), u\neq v}\ \sum_{X\in \Uu} \frac{\wei_G(X)^2}{M^2}&& \nonumber\\
            & \leq \binom{h}{2}\cdot \frac{W}{4h^2}\cdot \frac{1}{M^2}\cdot \sum_{X\in \Uu} \wei_G(X)&& \textrm{by }\wei_G(X)\leq \frac{W}{4h^2}\textrm{ for each }X\in \Uu,\nonumber\\
            & \leq \binom{h}{2}\cdot \frac{W}{4Mh^2}\leq \frac{1}{2} && \textrm{by }\sum_{X\in \Uu} \wei_G(X)=M\geq W/4. \label{eq:boundI}
 \end{align}

 We proceed to bounding the probability of $\overline{F}$.
 For a fixed $X \in V(\wh{G}')$ and $uv\in E(\Ddot{H})$ consider the event $A_{X,uv}$ defined as follows: $\phi(u)\neq \phi(v)$ and $X\in \pi(uv)$. We first  bound the probability of~$A_{X,uv}$:
 \begin{align*}
     \Pr[A_{X,uv}] & = \sum\limits_{(Y,Z) \in \Uu \times \Uu, Y \neq Z} \frac{\wei_G(Y)\wei_G(Z)}{M^2} \sum\limits_{P \in \Pp_{Y,Z}(\wh{G}'), X\in P} \frac{\lambda(P)}{\wei_G(Y)\wei_G(Z)}\\
     &= \frac{1}{M^2} \sum\limits_{(Y,Z) \in \Uu \times \Uu, Y \neq Z} \frac{\wei_G(Y)\wei_G(Z)}{\wei_G(Y)\wei_G(Z)} \sum\limits_{P \in \Pp_{Y,Z}(\wh{G}'), X\in P} \lambda(P)\\
     &=\frac{1}{M^2} \sum\limits_{P \in \Pp(\wh{G}'), X\in P} \lambda(P) = \frac{\gamma_{\wh{G}',\lambda}(X)}{M^2} \leqslant \frac{\gamma}{M^2}.
 \end{align*}

 Let us define the set of relevant pairs of edges of $\Ddot{H}$:
 \[\Pi\coloneqq \left\{\,(uv,u'v')\in E(\Ddot{H})\times E(\Ddot{H})\ \mid\ |\{u,v,u',v'\}|=4\,\right\}.\]
 Note that for any $(uv,u'v')\in \Pi$,  $(\phi(u), \phi(v), \pi(uv))$ and $(\phi(u'), \phi(v'), \pi(u'v'))$ are independent joint variables, and hence $A_{X,uv}$ and $A_{X',u'v'}$ are also independent events.
 Therefore, for all $X, X' \in V(\wh{G}')$,
 \begin{equation}\label{eq:bobr}
 \Pr[A_{X,uv} \cap A_{X',u'v'}] = \Pr[A_{X,uv}] \cdot \Pr[A_{X',u'v'}] \leqslant \frac{\gamma^2}{M^4} \leqslant 256 \cdot \frac{\gamma^2}{W^4}=\frac{1}{4h^2c\cdot n^{1+\eps}}.
 \end{equation}

 Observe that assuming $\phi$ is injective (that is, event $I$), that $(\phi,\pi)$ is a~crude $3$-fat model of $\Ddot{H}$ is equivalent to the following assertion: for all $(uv,u'v')\in \Pi$ and $(X,X')\in \Xi$, we do not have $X\in \pi(uv)$ and $X'\in \pi(u'v')$ simultaneously. Therefore, we have the following inclusion of events:
 \[\overline{F}\cap I\quad\subseteq\quad \bigcup_{(uv,u'v')\in \Pi}\ \bigcup_{(X,X')\in \Xi}\  A_{X,uv}\cap A_{X',u'v'}.\]
 Hence, we get
 \begin{align}
 \Pr[\overline{F}\cap I] &\leq |\Pi|\cdot |\Xi|\cdot \frac{1}{4h^2c\cdot n^{1+\eps}} && \textrm{by union bound and \eqref{eq:bobr},}\nonumber\\
 & \leq \frac{1}{4}  && \textrm{by \cref{clm:few-bad-pairs} and $|\Pi|\leq h^2$.} \label{eq:boundFI}
 \end{align}

 By combining \eqref{eq:boundI} and \eqref{eq:boundFI} we conclude that 
 $\Pr[F]\geq \Pr[F\cap I]\geq \frac{1}{4}$; or in other words, $(\phi,\pi)$ is a~crude $3$-fat model of $\Ddot{H}$ in $\wh{G}$ with probability at least $\frac{1}{4}$. By \cref{lem:d-fat-almost}, this model can be turned into a $3$-fat model $(\{\wh{B}_u\}_{u\in V(H)},\{\wh{B}_e\}_{e\in E(H)})$ of $H$. Finally, since the distances between the elements of $\Xx$ in $\wh{G}$ are not larger than in $G$, we conclude that taking
 \[B_f\coloneqq G\left[\bigcup_{X\in V(\wh{B}_f)} X\right]\qquad\textrm{for all }f\in V(H)\cup E(H)\]
 yields a $3$-fat model of $H$ in $G$.

 Taking into account the probability that any of the (randomized) algorithms of \cref{lem:partition,lem:LR-reformulation} returns an incorrect output, which can be bounded by $\frac{1}{8}$, we conclude that the presented algorithm finds either an $(\Oh(\|H\|^2\cdot n^{\frac{1}{2}+\eps}),\Oh(1/\eps))$-coverable balanced separator of $G$ or a $3$-fat model of $H$ in $G$ with probability at least $\frac{1}{8}$. As usual, the probability of success can be boosted to $1-\frac{1}{2^{n^{\Oh(1)}}}$ by making $n^{\Oh(1)}$ independent repetitions of the algorithm.
\end{proof}

%% file: inducedproof.tex
\section{Proofs of \cref{thm:star-separator,thm:induced-biclique}}

In this section we prove our results for graphs excluding an induced minor: \cref{thm:star-separator,thm:induced-biclique}. We need some terminology. 
A~\EMPH{star partition} of a~graph $G$ is a connected~partition $\Xx$ of $G$ such that for every $X \in \Xx$, there is a~vertex $v_X \in X$ satisfying $X \subseteq N[v_X]$.
In other words, it is a connected partition of strong radius at~most~1. 

\begin{lemma}\label{lem:star-contraction}
  Every $n$-vertex graph $G$ admits a~star partition $\Xx$ such that $G/\Xx$ has $\Oh(n^{4/3} \log^{2/3} n)$ edges. 
\end{lemma}

\begin{proof}
  We set $k \coloneqq \lceil n^{1/3} \log^{2/3} n \rceil$, and assume that $k \geqslant 2$ without loss of generality.
  Let $v_1, \ldots, v_h$ be a maximal sequence of distinct vertices such that for every $i \in [h]$, $v_i$ has at~most~$k$ neighbors in $G-\{v_1, \ldots, v_{i-1}\}$.
  Let $M \coloneqq \{v_1, \ldots, v_h\}$, and let $E_M$ be the set of edges of~$G$ with at~least one endpoint in~$M$.
  By construction, $|E_M| \leqslant k|M| \leqslant kn$.

  By the maximality of~$M$, the graph $G' \coloneqq G-M$ has minimum degree greater than~$k$.
  Thus $G'$ admits a~dominating set of size at~most~$\frac{2(n-h) \log k}{k}$~\cite{Irnautov74}.
  In particular, $G'$ admits a~star partition $\Xx'$ with at~most~$\frac{2(n-h) \log k}{k}$ parts.
  We extend it to a~star partition $\Xx$ of~$G$ by adding every singleton $\{v_1\}, \ldots, \{v_h\}$.

  The number of edges of $G/\Xx$ is at most \[|E_M|+{\frac{2(n-h) \log k}{k} \choose 2} \leqslant kn + \frac{4n^2 \log^2 k}{k^2} = \Oh(n^{4/3} \log^{2/3} n). \qedhere\]
\end{proof}

Then \cref{thm:star-separator} is an easy consequence of combining \cref{lem:star-contraction,thm:indminor-sep}.

\begin{proof}[Proof of \cref{thm:star-separator}]
Apply \cref{lem:star-contraction} to $G$, yielding a star partition $\Xx$ of $G$ such that the quotient graph $G/\Xx$ has $\Oh(n^{4/3} \log^{2/3} n)$ edges.
Noting that $G/\Xx$ also excludes $H$ as an induced minor, we may apply \cref{thm:indminor-sep} to $G/\Xx$, thus obtaining a balanced separator $\cal S$ of $G/\Xx$ of size $\Oh_H(\sqrt{m}\log m)=\Oh_H(n^{2/3}\log^{4/3} n)$. Then $\bigcup \cal S$ is a balanced separator in $G$ that is $(\Oh_H(n^{2/3}\log^{4/3} n),1)$-coverable.
\end{proof}

For \cref{thm:induced-biclique}, we use the following result of Chudnovsky et al.~\cite{logMenger} together with a classic consequence of a result of K\"uhn and Osthus~\cite{KuhnO04a}.

\begin{theorem}[{\cite[Theorem~9.1]{logMenger}}]
    Let $G$ be a $K_{t,t}$-induced-minor-free graph, for some $t\in \N$. Then $G$ admits a connected partition $\Xx$ in which every part is $(1,4)$-coverable and such that the graph $G/\Xx$ has chromatic number at most $t\cdot 2^t$.
\end{theorem}

\begin{theorem}[follows from {\cite[Theorem~1]{KuhnO04a}}]
    For every $s\in \N$, every $n$-vertex graph $G$ that excludes $K_s$ as a subgraph and $K_{s,s}$ as an induced minor has $\Oh_s(n)$ edges.
\end{theorem}

As observed in~\cite[Proposition~10.3]{logMenger}, the combination of these two results immediately yields the following. We note that in \cite{logMenger}, they use more recent, stronger results instead of the standard theorem of K\"uhn and Osthus~\cite{KuhnO04a}, but the latter suffices for our purpose.

\begin{lemma}\label{lem:biclique-clustering}
    For every $t\in \N$, every $n$-vertex $K_{t,t}$-induced-minor-free graph $G$ has a connected partition $\Xx$ such that every part of $\Xx$ is $(1,4)$-coverable and the quotient graph $G/\Xx$ has $\Oh_t(n)$ edges. 
\end{lemma}

Now \cref{thm:induced-biclique} follows from combining \cref{lem:biclique-clustering,thm:indminor-sep} in exactly the same manner as in the proof of~\cref{thm:star-separator}.

%% file: conjectures.tex
\section{Conjectures}\label{sec:conj}

Our work raises multiple questions about the existence of coverable balanced separators in fat-minor-free and induced-minor-free graphs. The most prominent one is whether \cref{conj:fatminor-sep} holds, or at least improving the additional $\Oh(n^\eps)$ factor to a polylogarithmic function of $n$. The only source of the $\Oh(n^\eps)$ factor is the clustering scheme of \cref{lem:partition}, hence we propose the following.

\begin{conjecture}\label{conj:better-clustering}
	For every graph $H$ there are constants $p,d\in \N$ such that every graph $G$ that excludes $H$ as a $3$-fat minor admits a connected partition $\Xx$ of weak diameter at most $d$ such that for every $u\in V(G)$, $B_G(u,2)$ intersects at most $p$ clusters of $\Xx$.
\end{conjecture}

We note that resolving \cref{conj:better-clustering} would not automatically resolve \cref{conj:fatminor-sep} because of the additional $\log n$ factor stemming from the gap between multicommodity flows and balanced cuts (\cref{thm:LR}). However, proving a weaker form of \cref{conj:better-clustering} with $p$ replaced by $\Oh(\log^\alpha n)$, for some $\alpha\in \Rp$, would directly translate to the weaker form of \cref{conj:fatminor-sep} with $c$ replaced by $\Oh(\log^{\alpha+1} n)$. We note that in \cref{conj:better-clustering}, it is not at all clear that the condition ``every ball of radius $2$ intersects few parts of~$\Xx$'' is the only correct one; other forms of asserting the sparsity of $\Xx$ may also be useful.

Noting that containing $H$ as a $2$-fat minor entails containing $H$ as an induced minor, \cref{thm:1/2-bound} implies that any $n$-vertex $H$-induced-minor-free graph has an~$(\Oh_H(n^{1/2+\eps}),\Oh(1/\eps))$-coverable balanced separator, for any $\eps>0$. As witnessed by \cref{thm:star-separator,thm:induced-biclique}, the induced-minor-free case may permit a~broader range of tools, hence we state the following weaker form of \cref{conj:fatminor-sep}.

\begin{conjecture}\label{conj:indminor-sep}
	For every graph $H$ there exist constants $c,r\in \N$ such that every $n$-vertex weighted graph~$G$ that excludes $H$ as an induced minor has a balanced separator that is $(c\sqrt{n},r)$-coverable.
\end{conjecture}

It would be particularly elegant if $r=1$ sufficed in \cref{conj:indminor-sep}, and \cref{thm:star-separator} gives an indication that this could indeed be the case. Towards this goal, we propose the following strengthening of \cref{lem:star-contraction} as a first step.

\begin{conjecture}
	For every graph $H$ there exist constants $c,q\in \N$ such that every $n$-vertex graph $G$ that excludes $H$ as an induced minor admits a connected  partition $\Xx$ such that every cluster of $\Xx$ is $(q,1)$-coverable and the quotient graph $G/\Xx$ has at most $c\cdot n$ edges.
\end{conjecture}

Again, a weaker form of this statement, with $c$ or $q$ (or both) replaced by a subpolynomial function of $n$, would also be interesting.







%% file: ref.bib
@article{GeorgakopoulosP25,
  author       = {Agelos Georgakopoulos and
                  Panos Papasoglu},
  title        = {Graph Minors and Metric Spaces},
  journal      = {Combinatorica},
  volume       = {45},
  number       = {3},
  pages        = {33},
  year         = {2025},
  url          = {https://doi.org/10.1007/s00493-025-00150-6},
  doi          = {10.1007/S00493-025-00150-6},
  bibsource    = {dblp computer science bibliography, https://dblp.org}
}

@inproceedings{KL24,
  title={Induced-minor-free graphs: Separator theorem, subexponential algorithms, and improved hardness of recognition},
  author={Korhonen, Tuukka and Lokshtanov, Daniel},
  booktitle={2024 Annual ACM-SIAM Symposium on Discrete Algorithms, SODA 2024},
  pages={5249--5275},
DOI = {10.1137/1.9781611977912.188},
  year={2024}
}

@article{DvorakN19,
  author       = {Zdenek Dvo\v{r}{\'{a}}k and
                  Sergey Norin},
  title        = {Treewidth of graphs with balanced separations},
  journal      = {Journal of Combinatorial Theory {B}},
  volume       = {137},
  pages        = {137--144},
  year         = {2019},
  url          = {https://doi.org/10.1016/j.jctb.2018.12.007},
  doi          = {10.1016/J.JCTB.2018.12.007},
  bibsource    = {dblp computer science bibliography, https://dblp.org}
}

@inproceedings{Lee17,
  author       = {James R. Lee},
  title        = {Separators in Region Intersection Graphs},
  booktitle    = {8th Innovations in Theoretical Computer Science Conference, {ITCS}
                  2017},
  series       = {LIPIcs},
  pages        = {1:1--1:8},
  publisher    = {Schloss Dagstuhl --- Leibniz-Zentrum f{\"{u}}r Informatik},
  year         = {2017},
  url          = {https://doi.org/10.4230/LIPIcs.ITCS.2017.1},
  doi          = {10.4230/LIPICS.ITCS.2017.1},
  bibsource    = {dblp computer science bibliography, https://dblp.org}
}

@article{KuhnO04a,
  author       = {Daniela K{\"{u}}hn and
                  Deryk Osthus},
  title        = {Induced Subdivisions In ${K}_{s,s}$-Free Graphs of Large
                  Average Degree},
  journal      = {Combinatorica},
  volume       = {24},
  number       = {2},
  pages        = {287--304},
  year         = {2004},
  url          = {https://doi.org/10.1007/s00493-004-0017-8},
  doi          = {10.1007/S00493-004-0017-8},
  bibsource    = {dblp computer science bibliography, https://dblp.org}
}

@article{CoarseBalSeps,
	author       = {Tara Abrishami and
	Jadwiga Czy\.zewska and
	Kacper Kluk and
	Marcin Pilipczuk and
	Micha\l{} Pilipczuk and
	Pawe\l{} Rzą\.zewski},
	title        = {On coarse tree decompositions and coarse balanced separators},
	journal      = {ArXiv preprint},
	volume       = {abs/2502.20182},
	url          = {https://doi.org/10.48550/arXiv.2502.20182},
	year         = {2025},
	doi = {10.48550/arXiv.2502.20182}
}

@article{logMenger,
	author       = {Maria Chudnovsky and Julien Codsi and Ajaykrishnan E. S. and Daniel Lokshtanov},
	title        = {Induced Minors and Coarse Tree Decompositions},
	journal      = {ArXiv preprint},
	volume       = {abs/2603.11379},
	url          = {https://doi.org/10.48550/arXiv.2603.11379},
	year         = {2026},
	doi = {10.48550/arXiv.2603.11379}
}

@article{lipton1979separator,
author = {Lipton, Richard J. and Tarjan, Robert Endre},
title = {A Separator Theorem for Planar Graphs},
journal = {SIAM Journal on Applied Mathematics},
volume = {36},
number = {2},
pages = {177-189},
year = {1979},
doi = {10.1137/0136016},
URL = {https://doi.org/10.1137/0136016},
eprint = {https://doi.org/10.1137/0136016}
}

@article{alon1990separator,
  title={A separator theorem for nonplanar graphs},
  author={Alon, Noga and Seymour, Paul and Thomas, Robin},
  journal={Journal of the American Mathematical Society},
  volume={3},
  number={4},
  pages={801--808},
  year={1990},
  publisher={American Mathematical Society},
  doi={10.1090/S0894-0347-1990-1065053-0},
  url={https://www.ams.org/journals/jams/1990-03-04/S0894-0347-1990-1065053-0/}
}

@article{GridCounter,
	author       = {Sandra Albrechtsen and James Davies},
	title        = {Counterexample to the conjectured coarse grid theorem},
	journal      = {ArXiv preprint},
	volume       = {abs/2508.15342},
	year         = {2025},
	url          = {https://doi.org/10.48550/arXiv.2508.15342},
	doi          = {10.48550/ARXIV.2508.15342},
	eprinttype    = {arXiv},
	eprint       = {2508.15342}
}

@article{SmallCounter,
	author       = {Sandra Albrechtsen and Marc Distel and Agelos Georgakopoulos},
	title        = {Small counterexamples to the fat minor conjecture},
	journal      = {ArXiv preprint},
	volume       = {abs/2601.05761},
	year         = {2026},
	url          = {https://doi.org/10.48550/arXiv.2601.05761},
	doi          = {10.48550/ARXIV.2601.05761},
	eprinttype    = {arXiv},
	eprint       = {2601.05761}
}

@article{DaviesHIM24,
  author = {Davies,  James and Hickingbotham,  Robert and Illingworth,  Freddie and McCarty,  Rose},
  title = {Fat minors cannot be thinned (by quasi-isometries)},
	journal      = {ArXiv preprint},
	volume       = {abs/2405.09383},
	year         = {2024},
	url          = {https://doi.org/10.48550/arXiv.2405.09383},
	doi          = {10.48550/ARXIV.2405.09383},
	eprinttype    = {arXiv},
	eprint       = {2405.09383}
}

@article{Filtser24,
  author       = {Arnold Filtser},
  title        = {Scattering and Sparse Partitions, and Their Applications},
  journal      = {{ACM} Transactions on Algorithms},
  volume       = {20},
  number       = {4},
  pages        = {30:1--30:42},
  year         = {2024},
  url          = {https://doi.org/10.1145/3672562},
  doi          = {10.1145/3672562},
  bibsource    = {dblp computer science bibliography, https://dblp.org}
}

@article{Leighton99,
  author       = {Frank Thomson Leighton and
                  Satish Rao},
  title        = {Multicommodity max-flow min-cut theorems and their use in designing
                  approximation algorithms},
  journal      = {Journal of the {ACM}},
  volume       = {46},
  number       = {6},
  pages        = {787--832},
  year         = {1999},
  url          = {https://doi.org/10.1145/331524.331526},
  doi          = {10.1145/331524.331526},
  bibsource    = {dblp computer science bibliography, https://dblp.org}
}

@article{Irnautov74,
  title={Estimation of the exterior stability number of a graph by means of the minimal degree of the vertices},
  author={Arnautov, Vladimir I.},
  journal={Prikladnaya Matematika i Programmirovanie},
  volume={11},
  number={3-8},
  pages={126},
  year={1974}
}
